\documentclass[12pt,a4paper]{amsart}
\usepackage{mathrsfs}
\usepackage{amsfonts}

\usepackage[active]{srcltx}
\usepackage{enumerate}

\usepackage{amsmath,amssymb,xspace,amsthm}
\usepackage[all]{xy}

 \usepackage[colorlinks,final,backref=page,hyperindex]{hyperref}

\newtheorem{theorem}{Theorem}[section]

\newtheorem{remark}[theorem]{Remark}
\newtheorem{lemma}[theorem]{Lemma}
\newtheorem{proposition}[theorem]{Proposition}
\newtheorem{corollary}[theorem]{Corollary}

\newtheorem{definition}[theorem]{Definition}
\numberwithin{equation}{section} \theoremstyle{definition}
\newcommand\darrowa{\longrightarrow {\mkern -27mu} {\raise 6pt \hbox{ $\pi_0$}} {\mkern 16mu}}

\newcommand{\C}{\ensuremath{\mathbb C}\xspace}

\newcommand{\Z}{\ensuremath{\mathbb{Z}}\xspace}

\def\mn{\mathfrak{n}}

\def\mg{\mathfrak{g}}
\def\mh{\mathfrak{h}}

\def\sl{\mathfrak{sl}}
\def\gl{\mathfrak{gl}}

\newcommand{\p}{\partial}

\def\p{\partial}
\def\bm{\mathbf{m}}
\def\bs{\mathbf{s}}
\def\ba{\mathbf{a}}
\def\br{\mathbf{r}}
\def\be{\mathbf{e}}
\def\b1{\mathbf{1}}

\def\cd{\mathcal{D}}

\begin{document}

\title[Category equivalence]{A category equivalence on   the Lie algebra of polynomial vector fields}
\author{  Genqiang Liu, Yufang Zhao }
\date{}\maketitle

\begin{abstract}For any positive integer $n$, let $W_n=\text{Der}(\mathbb{C}[t_1,\dots,t_n])$. The subspaces $\mathfrak{h}_n=\text{Span}\{t_1\frac{\partial}{\partial{t_1}},\dots,t_n\frac{\partial}{\partial{t_n}}\}$
 and $\Delta_n=\text{Span}\{\frac{\partial}{\partial{t_1}},\dots,\frac{\partial}{\partial{t_n}}\}$ are two abelian subalgebras of $W_n$.
We show that a full subcategory  $\Omega_{\mathbf{1}}$ of  the category of $W_n$-modules $M$ which are locally finite over $\Delta_n$
 is equivalent to some full subcategory of weight  $W_n$-modules $M$ which are cuspidal modules when restricted to the subalgebra $\mathfrak{sl}_{n+1}$ of $W_n$.
 \end{abstract}

\vskip 5pt
\noindent
{\em 2020  Math. Subj. Class.:}
17B10, 17B20, 17B65, 17B66, 17B68

\vskip 10pt

\section{Introduction}

The original motivation for this paper is to find relationship between non-weight modules and weight modules over Lie algebras. In \cite{BM}, Batra and Mazorchuk  introduced Whittaker pairs.  A pair $(\mg,\mn)$ is called a Whittaker pair if $\mn$ is a quasi-nilpotent subalgebra of
the Lie algebra $\mg$ and the adjoint action of $\mn$ on the $\mn$-module $\mg/\mn$ is locally nilpotent. For a pair $(\mg,\mn)$, we can  study $\mg$-modules that are locally finite over $\mn$.  Apart from Kostant's Whittaker modules (see \cite{K}), this kind of modules also include classical weight
modules  when $\mn$ is chosen to be a Cartan
subalgebra of $\mg$, and classical Harish-Chandra and generalized Harish-Chandra modules, see e.g.  \cite{PZ}. Let  $\mathfrak{W}^{\mg}_\mn$
 be the category which consists of all $\mg$-modules, on which the action of $\mn$  is locally finite.  A block decomposition of $\mathfrak{W}^{\mg}_\mn$ and
certain properties of simple Whittaker modules under some rather
mild assumptions were described in \cite{BM}. We want to find possible  equivalence between $\mathfrak{W}^{\mg}_\mn$ and $\mathfrak{W}^{\mg}_{\mn'}$ for two different
pairs $(\mg,\mn)$, $(\mg,\mn')$. For example, if $\mh$ and $\mh'$ are two Cartan subalgebras
of a complex semisimple Lie algebra $\mg$, then by the  conjugate relationship
between $\mh$ and $\mh'$, $\mathfrak{W}^{\mg}_\mh$ and $\mathfrak{W}^{\mg}_{\mh'}$ are equivalent.

Among the theory of infinite dimensional Lie algebras, the derivation Lie algebra    $W_n$ of $A_n=\C[t_1,\dots, t_n]$ is an important simple Lie algebra, see \cite{Kac,M1}. The subspaces $\mh_n:=\text{Span}\{t_1\frac{\partial}{\partial{t_1}},\dots,t_n\frac{\partial}{\partial{t_n}}\}$
 and $\Delta_n:=\text{Span}\{\frac{\partial}{\partial{t_1}},\dots,\frac{\partial}{\partial{t_n}}\}$ are two abelian subalgebras of $W_n$. A $W_n$-module $M$ with diagonalizable action of $\mh_n$ is called a weight module.
Tensor modules $T(P, V) $ over $W_n$ were introduced by Shen and Larsson, see \cite{Sh,La}, for a $\cd_n$-module $P$ and $\gl_n$-module V, where $\cd_n$ is the Weyl algebra. In \cite{LLZ}, the simplicity  of any $T(P, V ) $ was completely determined. When $V$ is finite dimensional and $P=\C[t_1^{\pm 1},\dots,t_n^{\pm 1}]$, the simplicity   of $T(P, V ) $ over the Witt algebra
on an $n$-dimensional torus  was given by Rao, see \cite{E1}. Weight modules with finite dimensional weight spaces are also called Harish-Chandra modules, see \cite{M1}.
The classification of simple Harish-Chandra modules over $W_1$ was given in \cite{M1}. Recently,
built  on the classification of simple uniformly bounded $W_n$ modules (see \cite{CG} for $n =2$  and \cite{XL}
for any $n$) and the characterization  of the weight set of simple weight $W_n$ modules
(see \cite{PS}), Grantcharov and Serganova completed the classification of all simple Harish-Chandra $W_n$-modules, see \cite{GS}. They proved that every such nontrivial module $M$ is a simple sub-quotient of some tensor module $T(P, V) $, where $P$ is a simple weight $\cd_n$-module.
In \cite{ZL}, under a finite condition, we show that  every  simple $W_n$-module $M$ with locally finite action of $\Delta_n$ is a simple quotient of some tensor module $T(A_n^{\ba}, V) $, where $A_n^{\ba}$ is a simple $\cd_n$-module twisted from $A_n$ by an isomorphism, see (\ref{A-mod}). From these known results, it seems that there is a relation between a subcategory  of $W_n$-modules $M$ which are locally finite over $\Delta_n$
and  some  subcategory of  uniformly bounded weight  $W_n$-modules. In this paper, we will address this problem.

 The paper is organized as follows. In Section 2, we collect all necessary preliminaries, including the Whittaker category $\Omega_{\mathbf{1}}$ and the cuspidal category $\mathcal{C}_{\alpha}$. The category $\Omega_{\mathbf{1}}$ is  the category consisting of $W_n$-modules $M$ such that $\p_i-1$ acts locally nilpotently on $M$ for any $i\in\{1,\dots,n\}$, and the subspace $\text{wh}_{\b1}(M)=\{v\in M \mid \p_i v=v, \  i=1,\dots,n \}$ is finite dimensional.
 A weight module $M$ with finite dimensional weight spaces over  $W_n$ is called a cuspidal module if it is cuspidal when restricted to the subalgebra $\sl_{n+1}$, see Definition \ref{cusp-def}.  For an $\alpha\in\C^n$, $\mathcal{C}_{\alpha}$ is the category of cuspidal $W_n$-modules $M$ such that $\text{Supp}M\subset \mathbf{\alpha}+\Z^n$.
  In Section 3, we study the localization $U_\p$ of $U(W_n)$ with respect to the Ore subset generated by $\p_1,\dots,\p_n$. Let $H_n$ be the centralizer of $\mh_n\oplus\Delta_n$ in $U_\p$.
 We show that $H_n$ is a tensor product factor of $U_\p$,  and $H_n$ is isomorphic to the opposite algebra of the endomorphism algebra of a universal Whittaker module $Q_\b1$, see Theorems \ref{tensor-iso} and  \ref{endo}. In Section 4, using $H_n$-modules, we
 study the categories $\Omega_{\b1}$ and $\mathcal{C}_{\alpha}$.
Every dominant integral weight  $\gamma$ of $\gl_n$ can define a full subcategory $\Omega_{\b1}^{[\gamma]}$  of $\Omega_{\b1}$ and a subcategory $\mathcal{C}^{[\gamma]}_\alpha$  of $\mathcal{C}_\alpha$. We  show that
  there exists a suitable $\alpha\in \C^n$ such that  $\Omega_{\b1}^{[\gamma]}$  is  equivalent to $\mathcal{C}^{[\gamma]}_\alpha$, see Proposition \ref{equ-hw} and Theorems \ref{H-W} and \ref{main-th}.

We denote by $\Z$, $\Z_{> 0}$, $\Z_{\geq 0}$, $\C$ and $\C^*$ the sets of integers, positive integers, nonnegative
integers, complex numbers, and nonzero complex numbers, respectively.  For a Lie algebra
$L$ over $\C$, we use $U(L)$ to denote the universal enveloping algebra of $L$.

\section{Preliminaries and known results}

\subsection{Witt algebras $W_n$}
Throughout this paper, $n$ denotes a positive integer.
We fix the vector space $\mathbb{C}^n$ of $n$-dimensional complex vectors.
Denote its standard basis by $\{\be_1,\be_2,\dots,\be_n\}$.  Let $|\bm|=m_1+\cdots+m_n$ for any $\bm=(m_1,\dots,m_n)\in\C^n$.

Let $A_n=\C[t_1,t_2,\dots,t_n]$ be the polynomial algebra in the commuting variables $t_1,t_2,\cdots,t_n$.
Let $\p_{i}=\frac{\p}{\p{t_i}}, h_i=t_i\partial_i$ for any $i=1, 2, \cdots, n$.
For the convenience,  set
$t^\bm=t_1^{m_1}\cdots t_n^{m_n}$, $h^{\bm}=h_1^{m_1}\cdots h_n^{m_n}$ for any $\bm=(m_1,\dots,m_n)\in\Z_{\geq 0}^n$.
Recall that  the Witt algebra $ W_n=\text{Der}(A_n)=\oplus_{i=1}^n{ A}_n\p_{i}$ is a free $A_n$-module of rank $n$, which is also called the Lie algebra of polynomial vector fields on $\C^n$.
It  has the following Lie bracket:
$$[t^{\bm}\p_{i},t^{\br}\p_{j}]=
r_it^{\bm+\br-\be_i}\p_{j}
-m_jt^{\bm+\br-\be_j}\p_{i},$$
where all $\bm,\br\in \Z_{\geq 0}^n, i, j=1,\dots,n$.

It is well-known that  $W_n$ is a simple Lie algebra.
Let $\mathfrak{h}_n=\text{Span}\{h_1,\dots , h_n \}$ which is a Cartan subalgebra (a maximal
abelian subalgebra that is diagonalizable on $W_n$ with respect to the
adjoint action) of $W_n$. Denote $\Delta_n=\text{Span}\{\p_1,\dots,\p_n\}$ which is also an $n$-dimensional abelian subalgebra of $W_n$.

\begin{definition}A $W_n$-module $M$ is called a weight module if the action of $\mh_n$ on $M$
is diagonalizable, i.e.,
$M=\bigoplus_{\lambda \in \C^n}M_{\lambda},$ where
$$M_{\lambda}=\{v \in M\mid h_iv=\lambda_iv,  \forall \ i=1,\dots,n\}.$$
\end{definition}

For a weight $W_n$-module $M$, denote its weight set by $\text{Supp}M=\{\lambda \in \C^n\mid M_\lambda\neq 0\}$. For example, the adjoint module $W_n$ has the weight set  $\text{Supp}W_n=\Z^n$. If $M$ is indecomposable, then there is a $\lambda \in \C^n$ such that $\text{Supp}M\subset \lambda+\Z^n$.

\subsection{ Finite-dimensional $\mathfrak{gl}_n$-modules}
We denote by $E_{ij}$ the $n\times n$ square matrix
with $1$ as its $(i,j)$-entry and 0 as  other entries. Then $\{E_{ij}\mid 1\leq i, j\leq n\}$ is a basis of the general linear  Lie
algebra $\gl_n$.

We can view  a $\lambda\in\C^n$ as  a linear function on $\oplus_{i=1}^n \C E_{ii}$ such that $\lambda(E_{ii})=\lambda_i$ for all $i\in \{1,\dots,n\}$.
Let $\Lambda_+=\{\lambda\in\C^n\,|\,\lambda_i-\lambda_{i+1}\in\Z_{\geq 0} \text{ for } i=1,2,\dots,n-1\}$ be the set of dominant integral weights of $\gl_n$.
It is well known that the simple highest weight $\gl_n$-module $V(\lambda)$ is finite dimensional if and only if $\lambda\in \Lambda_+$.  Define the fundamental weights $\delta_i\in\Lambda_+$ by
$\delta_i=\be_1+\cdots+\be_i$ for all $i=1,2,\dots, n$.
For convenience, we define $\delta_0=0\in \Lambda_+$. So $V(\delta_0)$ is  the $1$-dimensional trivial $\gl_n$-module.
A module $V$ over a Lie algebra $\mg$ is called a trivial module if $X v=0$ for any $X\in \mg, v\in V$.
 It is well-known that the module $V(\delta_1)$ is isomorphic to the
natural module $\mathbb{C}^n$ of $\gl_n$.
The exterior product $\bigwedge^k(\mathbb{C}^n)=\mathbb{C}^n\wedge\cdots\wedge
\mathbb{C}^n\ \ (k\ \mbox{times})$ is a $\mathfrak{gl}_n$-module
with the action given by $$X(v_1\wedge\cdots\wedge
v_k)=\sum\limits_{i=1}^k v_1\wedge\cdots v_{i-1}\wedge
Xv_i\cdots\wedge v_k, \,\,\forall \,\, X\in \gl_n.$$
Moreover, ${\bigwedge}^k(\mathbb{C}^n)\cong V(\delta_k), \forall\ 1\leq k\leq n$, see Exercise 21.11 in \cite{H}.
 We can set  $V(\delta_0)=\bigwedge^0(\mathbb{C}^n)=\C$ and
$v\wedge a=av$ for any $v\in\C^n, a\in\C$.

\subsection{Whittaker modules and cuspidal modules}

We  recall the definition of Whittaker pair introduced in \cite{BM}. A pair $(\mathfrak{g},\mathfrak{n})$ is called a Whittaker pair if $\mathfrak{n}$ is a quasi-nilpotent subalgebra (a generalization of the nilpotent algebra) of the Lie algebra $\mathfrak{g}$ and the adjoint action of $\mathfrak{n}$ on $\mathfrak{g}/\mathfrak{n}$ is locally nilpotent.
For the Lie algebra $W_n$, $\Delta_n=\text{Span}\{\p_1,\dots,\p_n\}$  is an abelian subalgebra, whence is a nilpotent subalgebra.
Since the adjoint action of $\Delta_n$ on $ W_n/\Delta_n$ is locally nilpotent, $(W_n,\Delta_n)$ is a Whittaker pair.
A $W_n$-module $M$ is called a   Whittaker module if
the action of $\Delta_n$ on $M$ is locally finite.

\begin{definition} For any $\mathbf{a}=(a_1,\cdots,a_n)\in \C^n$,  let $\Omega_{\mathbf{a}}$ be the category consisting of $W_n$-modules $M$ such that:
\begin{enumerate}[$($a$)$]
\item  For any $v\in M$, $1\leq i \leq n$, there is a $k\in \Z_{>0}$ such that $(\p_i-a_i)^kv=0$.
\item  The subspace $\emph{wh}_{\ba}(M)=\{v\in M \mid \p_i v=a_iv, \ \forall\ i \}$ is finite dimensional.
\end{enumerate}An element in $\emph{wh}_{\ba}(M)$ is called a Whittaker vector. For a vector $\ba\in \C^n$, we call it non-singular if $a_i\neq 0$ for any $i\in \{1,\cdots,n\}$.
\end{definition}

Denote $\b1=(1,\dots,1)\in \C^n$.
For a $W_n$-module $M$, and  an automorphism $\sigma$ of $W_n$, $M$ can be twisted by $\sigma$ to  be a new $W_n$-module $M^\sigma$.
As a vector space $M^\sigma=M$, the action of $W_n$ on $M^\sigma$ is defined by $u\cdot v= \sigma(u)v, \forall\ u\in W_n, v\in M$.

\begin{lemma}If $\ba\neq 0$, then $\Omega_\ba$ is equivalent to $\Omega_\b1$.
\end{lemma}

\begin{proof} Any $X=(x_{ij})_{n\times n}\in GL(n,\C)$ can define an automorphism $\sigma_X$ of $W_n$ such that $\sigma_X(\p_i)=\sum_{j=1}^n x_{ij}\p_j$. Choosing  a proper $X$,  we can obtain that $\p_i-1$ acts locally nilpotently on $M^{\sigma_X}$, i.e., $M^{\sigma_X}\in \Omega_\b1$ for any $M\in \Omega_\ba$.
\end{proof}

Let $E_n=\sum_{i=1}^n t_i\p_i$. It is known that the subalgebra of $W_n$ spanned by $\p_i, t_i\p_j, t_iE_n, 1\leq i,j \leq n$ is isomorphic to $\sl_{n+1}$, see page 578 in \cite{M}. We denote this subalgebra still by $\sl_{n+1}$. A weight module $M$ with finite dimensional weight spaces over  $\sl_{n+1}$ is called a cuspidal module if every root vector of $\sl_{n+1}$ acts injectively on $M$, see \cite{M,GS1}.

\begin{definition}\label{cusp-def}  A weight module $M$ with finite dimensional weight spaces over  $W_n$ is called a cuspidal module if it is cuspidal when restricted to the subalgebra $\sl_{n+1}$, i.e., $\p_i, t_i\p_j, t_iE_n$ act injectively on $M$, for any $1\leq i\neq j \leq n$. For an $\alpha=(\alpha_1,\dots,\alpha_n)\in\C^n$, let $\mathcal{C}_{\alpha}$ be the category of cuspidal $W_n$-modules $M$ such that $\emph{Supp}M\subset \mathbf{\alpha}+\Z^n$.
\end{definition}

In the  remainder of   this section, we give all simple modules in $\Omega_{\b1}$ and $\mathcal{C}_\alpha$ using the results in \cite{XL,ZL}.

\subsection{Shen-Larsson's modules}
The  Weyl algebra $\mathcal{D}_{n}$
 over $A_n$  is the unital  associative algebra
over $\mathbb{C}$ generated by $t_1,\dots,t_n$,
$\partial_1,\dots,\partial_n$ subject to
the relations
$$[\partial_i, \partial_j]=[t_i,t_j]=0,\qquad [\partial_i,t_j]=\delta_{i,j},\ 1\leq i,j\leq n.$$

The following homomorphism was well known, see \cite{Sh} or  \cite{LLZ}.

\begin{theorem}\label{mainth}
The linear map $\phi:U(W_n)\rightarrow \mathcal{D}_n\otimes U(\gl_n)$ such that
\begin{equation}\label{action}\aligned\phi(t^{\bm}\partial_k)&=
t^{\bm}\partial_k \otimes 1+\sum_{i=1}^n m_i t^{\bm-\be_i} \otimes E_{ik},\endaligned\end{equation}  is an associative algebra homomorphism.
\end{theorem}

The homomorphism  in Theorem \ref{mainth} implies  that there is  a close relationship between $W_n$-modules and $\cd_n$-modules.
For any simple $\mathcal{D}_n$-module $P$, and simple $\gl_n$-module $V$, the tensor product $P\otimes V$ can be lifted to a $W_n$-module denoted by $T(P,V)$  under  the map $\phi$.

The simplicity  of $T(P, V)$ was determined in \cite{LLZ}, for  any simple $\mathcal{D}_n$-module $P$, and simple $\gl_n$-module $V$.
 For a finite dimensional simple $\gl_n$-module $V$, $T(P, V)$ is a simple module over
$W_n$ if  $V \not\cong V(\delta_k)$ for any $k\in \{0, 1,\cdots, n\}$. Moreover
we have the following complex of $W_n$-modules:
\begin{equation}\label{complex}\xymatrix
{0\ar[r] &T(P,V(\delta_0))\ar[r]^{\pi^P_0} &T(P,V(\delta_1))\ar[r]^{\hskip 1cm\pi^P_1} &\cdots \\
\cdots\ar[r]^{\pi^P_{n-2}\hskip 2cm}&{T(P,V(\delta_{n-1}))}\ar[r]^{\hskip .8cm\pi^P_{n-1}} &T(P,V(\delta_n))\ar[r]& 0,}
\end{equation}
where
$\pi^P_{k-1}$ is the $W_n$-module homomorphism defined by
\begin{equation*}\begin{array}{lrcl}
\pi^P_{k-1}:& T(P,V(\delta_{k-1})) & \rightarrow & T(P, V(\delta_{k})),\\
      & p\otimes v & \mapsto & \sum_{j=1}^{n} \p_j\cdot p\otimes (\be_j\wedge v),
\end{array}\end{equation*}
for all $p\in P$ and $v\in V(\delta_{k-1})$, where $T(P, V(\delta_{-1})) =0$, $k\in \{0, 1,\cdots, n\}$.
By Theorem 3.5 in \cite{LLZ}, $\text{im}\pi^P_{0},\text{im}\pi^P_{1},\dots,\text{im}\pi^P_{n-1}$ are simple $W_n$-modules.  $T(P,V(\delta_0))$ is not simple if and only if $P\cong A_n$.
Moreover $T(P,V(\delta_n))$ is irreducible if and only if $\sum_{k=1}^n \partial_k P=P$. If $T(P,V(\delta_n))$ is reducible, then $T(P,V(\delta_n))/\text{im}\pi^P_{n-1}$ is a trivial module.

 For a simple $\mathcal{D}_n$-module $P$, and simple $\gl_n$-module $V(\lambda)$, we denote
$$L(P,V(\lambda)):=\left\{\begin{array}{l}
T(P,V(\lambda)), \ \text{if}\ \lambda\neq \delta_i, i=0,1,\dots, n,\\
\text{im}\pi^P_{i-1},\ \text{if}\ \lambda= \delta_i, i=1,\dots, n.
 \end{array}
   \right.$$ And $L(P,V(\delta_0)):=L(P,V(\delta_1))$.
Hence $L(P,V(\lambda))$ are simple $W_n$-modules, for any $\lambda\in \Lambda_+$.

\subsection{Simple modules in $\Omega_{\b1}$}
For an $\ba\in \C^n$, let $\sigma_{\ba}$ be the algebra automorphism of $\mathcal{D}_n$ defined by $$t_i\mapsto t_i, \p_i \mapsto \p_i+a_i.$$
The $\mathcal{D}_n$-module $A_n$ can be twisted by $\sigma_{\ba}$ to be a new $\mathcal{D}_n$-module $A_n^{\ba}$.  Explicitly the action of
$\mathcal{D}_n$ on  $A_n^{\ba}$ is defined by
\begin{equation}\label{A-mod}X\cdot f(t)=\sigma_{\ba}(X)f(t), \ \ f(t)\in A_n, X\in \mathcal{D}_n.\end{equation}

By Theorem 4.6 and Lemma 4.7 in \cite{ZL}, we have the following theorem.

\begin{theorem}\label{c-t} Suppose that $\ba\in \C^n$ is nonsingular, i.e., $a_i\neq 0$ for any $i=1,\dots, n$.  If $M$ is a simple module in $\Omega_{\mathbf{a}}$, then $M$ is isomorphic to  $L(A_n^{\ba}, V(\lambda))$ for some $\lambda\in\Lambda_+$.
\end{theorem}

For a $\gamma\in \Lambda_+$, let $[\gamma]=\{\lambda\in \Lambda_+\mid \lambda-\gamma\in \Z^n\}$,  $\Omega_{\b1}^{[\gamma]}$ be the Serre subcategory of $\Omega_{\b1}$ generated by the simple
modules $L(A_n^{\b1}, V(\lambda))$, for $\lambda\in [\gamma]$. In Section $4$,  we  will show   the subcategory decomposition
$\Omega_{\b1}=\bigoplus_{[\gamma]}\Omega_{\b1}^{[\gamma]}$.

\subsection{Simple modules in $\mathcal{C}_\alpha$}
For any $\mu\in\C^n$, let $P(\mu)=t_1^{\mu_1}\cdots t_n^{\mu_n}\C[t_1^{\pm 1},\dots,t_n^{\pm 1}]$. As a $\cd_n$-module, $P(\mu)$ is simple if and only if $\mu_i\not\in \Z$ for any  $i=1,\dots, n$. Since $\phi(h_k)=
t_k\partial_k \otimes 1+1 \otimes E_{kk}$, $\text{Supp}T(P(\mu),V(\lambda))\subset \mu+\lambda+\Z^n$.

\begin{lemma}\label{cuspidal}For $\lambda\in\Lambda_+$ and  $\mu\in\C^n$, $T(P(\mu), V(\lambda))$ is cuspidal if and only if   $\mu_i\not\in \Z, |\mu+\lambda|+\lambda_i\not\in \Z$ for any  $i=1,\dots, n$.
\end{lemma}

\begin{proof} By (\ref{action}), for any $v\in V(\lambda), \bm\in \Z^n$, we  have
$$\aligned \p_i(t^{\mu+\bm}\otimes v)=&\ (\mu_i+m_i)t^{\mu+\bm-\be_i}\otimes v,\\
t_i \partial_j(t^{\mu+\bm}\otimes v)= &\
(\mu_j+m_j)t^{\mu+\bm+\be_i-\be_j}\otimes v
+ t^{\mu+\bm}\otimes E_{ij}v,\ i\neq j,\\
t_iE_n(t^{\mu+\bm}\otimes v)= &\  t^{\mu+\bm+\be_i}\otimes (|\mu+\lambda+\bm|+E_{ii})v
+\sum_{j=1,j\neq i}^n t^{\mu+\bm+\be_j}\otimes  E_{ij}v.
\endaligned$$
Then it can be checked that $\p_i, t_i\p_j$ act injectively on $T(P(\mu), V(\lambda))$ if and only if  $\mu_i,\mu_j\not\in \Z$, and
$t_iE_n$ acts injectively on $T(P(\mu), V(\lambda))$ if and only if $|\mu+\lambda|+\lambda_i\not\in \Z$.
\end{proof}

The following theorem was shown in \cite{XL}.

\begin{theorem}If $M$ is a simple cuspidal $W_n$-module, then
 $M$ is isomorphic to  $L(P(\mu), V(\lambda))$ for some $\lambda\in\Lambda_+$ and  $\mu\in\C^n$ with $\mu_i\not\in \Z, |\mu+\lambda|+\lambda_i\not\in \Z$ for any  $i=1,\dots, n$.
\end{theorem}

 For a $\gamma\in \Lambda_+$,
 let $\mathcal{C}_{\alpha}^{[\gamma]}$ be the  Serre subcategory of $\mathcal{C}_{\alpha}$ generated by the simple
modules $L(P(\mu), V(\lambda))$, for $\lambda\in [\gamma], \mu\in \C^n$ with $\lambda+\mu\in \alpha+\Z^n$.
We will show that $\Omega_{\b1}^{[\gamma]}$ is equivalent to some $\mathcal{C}_{\alpha}^{[\gamma]}$ in Section 4.

\begin{remark}By Lemma \ref{cuspidal}, for a $\gamma\in \Lambda_+$, $\mathcal{C}_{\alpha}^{[\gamma]}$ is nonempty if and only if $\alpha_i-\gamma_i\not\in \Z, |\alpha|+\gamma_i\not\in \Z$ for any  $i=1,\dots, n$.
\end{remark}

\section{Localization of $U(W_n)$}

 Throughout the paper, we denote $U=U(W_n)$. In this section, we study the structure of $U$, mainly the centralizer of $\Delta_n\oplus \mh_n$ in a bigger algebra $U_{\p}$.

\subsection{Tensor product decomposition of  $U_{\p}$}

Since each $\text{ad}\p_i$ is locally nilpotent on $U$, the following set
   $$\p:=\{\p_1^{i_1}\dots \p_n^{i_n}\mid i_1,\dots, i_n\in\Z_{\geq0}\}$$ is an Ore subset of $U$, see Lemma 4.2 in \cite{M}.
Let  $U_{\p}$  be the localization of $U$ with respect to $\p$.
Define $$H_n=\{u\in U_{\p}\mid [u,\p_i]=[u, h_i]=0,\ \forall\ i=1,\dots,n\}, $$
which is a subalgebra of $U_{\p} $.   Define the following elements in $U_\p$:
 \begin{equation}\label{x-w-def}
 \aligned  z_{i,j}= &\  (t_i\p_j)\p_i\p_j^{-1}-t_i\p_i ,\\
  z_{i,l,j}=&\ (t_it_l\p_j)\p_i\p_l\p_j^{-1}-(t_i\p_j)\p_i\p_j^{-1}(t_l\p_l )\\
  & -(t_l\p_j)\p_l\p_j^{-1}(t_i\p_i)+(t_l\p_l )(t_i\p_i)+\delta_{il}t_i\p_l, \\
  z_i=&\  (t_i^3\p_i)\p_i^{2}-3(t_i^2\p_i)(t_i\p_i-1)\p_i+2(t_i\p_i)(t_i\p_i-1)(t_i\p_i-2),
  \endaligned\end{equation}
  where $i,j,l\in\{1,\dots,n\}$.

The following Lemma is used to construct generators of the associative algebra $H_n$.
\begin{lemma}
 For any $i,j,l\in\{1,\dots,n\}$,  $z_{i,j}, z_{i,l,j}, z_i\in H_n$.
\end{lemma}

\begin{proof}
It is easy to see that $[h_k,  z_{i,j} ]=[h_l, z_{i,j,k}]=[h_k,z_i]=0$.

We can compute that $$\aligned \   [\p_k, (t_i\p_j)\p_i\p_j^{-1}-t_i\p_i ]=\delta_{ki}\p_j\p_i\p_j^{-1}-\delta_{ki}\p_i=0,
\endaligned$$
and
$$\aligned \  &   [\p_k, z_{i,l,j}]\\
=&\  \delta_{ki}(t_l\p_j)\p_i\p_l\p_j^{-1}
+\delta_{kl}(t_i\p_j)\p_i\p_l\p_j^{-1}
-\delta_{ki}\p_j\p_i\p_j^{-1}(t_l\p_l )
-\delta_{kl}(t_i\p_j) \p_i\p_j^{-1}\p_l\\
&\  -\delta_{kl}\p_j\p_l\p_j^{-1}(t_i\p_i)
- \delta_{ki}(t_l\p_j)\p_l\p_j^{-1}\p_i
+\delta_{kl}\p_l(t_i\p_i)
+\delta_{ki}(t_l\p_l )\p_i
+\delta_{ki}\delta_{li}\p_l\\
=&-\delta_{ki}\p_i(t_l\p_l )
-\delta_{kl}\p_l(t_i\p_i)
+\delta_{kl}\p_l(t_i\p_i)
+\delta_{ki}(t_l\p_l )\p_i
+\delta_{ki}\delta_{li}\p_l\\
=&\ -\delta_{ki}\p_i(t_l\p_l )
+\delta_{ki}(t_l\p_l )\p_i
+\delta_{ki}\delta_{li}\p_l\\
=&\  0.
\endaligned$$

 Finally we can also  obtain that
$$\aligned \  &
 [\p_i, z_i]\\
=&\   [\p_i, (t_i^3\p_i)\p_i^{2}-3(t_i^2\p_i)(t_i\p_i-1)\p_i+2(t_i\p_i)(t_i\p_i-1)(t_i\p_i-2)]\\
=&\   3(t_i^2\p_i)\p_i^{2}-6(t_i\p_i)(t_i\p_i-1)\p_i-3(t_i^2\p_i)\p_i^{2}
+2\p_i(t_i\p_i-1)(t_1\p_i-2)\\
&\ +2(t_i\p_i)\p_i(t_i\p_i-2)
+2(t_i\p_i)(t_i\p_i-1)\p_i\\
=&\  -6(t_i\p_i)(t_i\p_i-1)\p_i+6(t_i\p_i)(t_i\p_i-1)\p_i\\
=&\  0.
\endaligned$$
\end{proof}
The following lemma is well known. For the convenience of proving Lemma \ref{H-generator}, we give its proof.
\begin{lemma}\label{generator}
When $n>1$, $W_n$ is generated $t^{\bm}\p_j$, $j=1,\dots, n$, $\bm\in\Z_{\geq 0}^n $ with $|\bm|=0,1,2$.
\end{lemma}
\begin{proof}
  For any $\bm\in\Z_{\geq 0}^n $ with $|\bm|\geq 3$. If $m_j=0$, then there must exists $i\neq j$ such that $m_i\neq 0$. From
\begin{equation}
\label{bracket}t^{\bm}\p_j=\left\{\begin{array}{l} [t^{\bm-\be_i}\p_j,t_jt_i\p_j],\text{if}\ m_j=0,\\
\frac{1}{(3-m_j)}[t^{\bm-\be_j}\p_j,t_j^2\p_j], \text{if}\ m_j\neq 0, 3, \\
\ [t_j\p_i, t_j^2t_i\p_j]+t_j^2t_i\p_i,\text{if}\  m_j= 3, |\bm|=3,\\
\ \frac{1}{2}[t^{\bm-2\be_j}\p_j,t_j^3\p_j],\text{if}\  m_j= 3, |\bm|>3,
 \end{array}
 \right.\end{equation}
 by induction on $|\bm|$, we can complete the proof.
\end{proof}

Let $B_n$ be the subalgebra of $U_{\p}$
generated by $\p_i, h_i, \p_i^{-1}, i=1,\dots,n$.
The following result gives a tensor product decomposition of $U_\p$.

\begin{theorem}\label{tensor-iso} $U_{\p} =B_nH_n\cong B_n\otimes H_n$.
\end{theorem}

\begin{proof} Clearly $B_nH_n \subseteq U_{\p}$. To show that $U_{\p} \subseteq B_nH_n$, we need to explain that the generators of $W_n$ belong to $B_nH_n$. When $n=1$, $W_1$ is generated by $\p_1, t_1\p_1, t_1^2\p_1, t_1^3\p_1$, and
$$\aligned & t_1^2\p_1=z_{1,1,1}\p_1^{-1}+(t_1\p_1)^2\p_1^{-1}-(t_1\p_1)\p_1^{-1}, \p_1, t_1\p_1 \in B_1H_1,\\
&  (t_1^3\p_1)=z_1\p_1^{-2}+3(t_1^2\p_1)(t_1\p_1-1)\p_1^{-1}
-2(t_1\p_1)(t_1\p_1-1)(t_1\p_1-2)\p_1^{-2}\in B_1H_1.  \endaligned$$
By Lemma \ref{generator}, when $n>1$, $W_n$ is generated $t^{\bm}\p_i$, $i=1,\dots, n$, $\bm\in\Z_{\geq 0}^n $ with $|\bm|=0,1,2$. These generators belong to
$B_nH_n$, since
$$\aligned t_i\p_j= & z_{i,j}\p_j\p_i^{-1}+(t _i\p_i)\p_j\p_i^{-1}\in B_nH_n,\\
t_it_l\p_j=&z_{i,l,j}\p_i^{-1}\p_l^{-1}\p_j+(t_i\p_j)\p_i\p_j^{-1}(t_l\p_l )\p_i^{-1}\p_l^{-1}\p_j
   +(t_l\p_j)\p_l\p_j^{-1}(t_i\p_i)\p_i^{-1}\p_l^{-1}\p_j\\
   &-(t_l\p_l )(t_i\p_i)\p_i^{-1}\p_l^{-1}\p_j-\delta_{il}(t_i\p_l)\p_i^{-1}\p_l^{-1}\p_j\in B_nH_n.
\endaligned$$
Since $[B_n,H_n]=0$,   we have that $B_nH_n\cong B_n\otimes H_n$.
\end{proof}

\subsection{Isomorphism between $H_n$ and the endomorphism algebra of a universal Whittaker module}

Let $\C_{\b1}=\C v_{\b1}$ be the one dimensional $U(\Delta_n)$-module such that $\p_i v_{\b1}= v_{\b1}$ for any $i\in\{1,\dots,n\}$.

Let $$Q_{\b1}=U(W_n)\otimes_{U(\Delta_n)} \C_{\b1},$$
an induced $W_n$-module, and
$$H_n^{\b1}=\text{End}_{W_n}(Q_{\b1})^{\text{op}},$$ an associative algebra over $\C$. Then $Q_{\b1}$ is both a left $U(W_n)$-module and a right  $H_n^{\b1}$-module. Let $H_n^{\b1}\text{-mod}$ be the category of finite dimensional $H_n^{\b1}$-modules.

In \cite{ZL}, we showed the following theorem.
\begin{theorem}\label{s-e}
The functors $M\mapsto \mathrm{wh}_{\b1}(M)$ and $V\mapsto Q_{\b1}\otimes_{H_{\b1}} V$ are inverse equivalence between  $\Omega_{\mathbf{1}}$ and $H_n^{\b1}\text{-mod}$ .
\end{theorem}

Inspired by the constructions of $z_{i,j}, z_{i,l,j}$ in (\ref{x-w-def}), we can give more useful elements in $H_n$.

\begin{lemma}\label{H-generator}For any $j\in \{1,\dots,n\}$ and $\bm\in \Z_{\geq 0}^n$ with $|\bm|\geq 1, \bm\neq \be_j$,  there is an $X_{\bm,j}\in H_n$ such that
\begin{equation} \label{x-operator} X_{\bm,j}=(t^\bm\p_j)\p^{\bm-\be_j}
+\sum_{0<|\br|<|\bm|}(t^\br\p_j)g_{\br}(h)\p^{\br-\be_j},\end{equation}
where $g_{\br}(h)$ is a polynomial  in $h_1,\cdots,h_n$ of the degree $|\bm|-|\br|$.
\end{lemma}
\begin{proof}

When $n=1$, denote $d_i:=t_1^{i+1}\partial_1$, for any $i\in \Z_{\geq -1}$.
Define
$$
X_{m+1,1}:=d_md_{-1}^{m}+\sum_{k=1}^{m-1} (-1)^{m-k}\binom{m+1}{k+1} d_k(\prod_{i=1}^{m-k}(d_0-i))d_{-1}^k+(-1)^m m\prod_{i=0}^{m}(d_0-i),$$ for any $m\in \Z_{\geq 1}$. It can be directly verified that $X_{m+1,1}\in H_1$.

When $n>1$, we do induction on $|\bm|$. If $|\bm|=1, 2$, then we set $X_{\be_i,j}=z_{i,j}$ and $X_{\be_i+\be_l,j}=z_{i,l,j}$, where $z_{i,j}, z_{i,l,j}$ are defined in (\ref{x-w-def}). For any $\bm\in\Z_{\geq 0}^n $ with $|\bm|>2$, there exists $m_i>0$. We set that
\begin{equation} X_{\bm, j}=\left\{\begin{array}{l} [X_{\bm-\be_i,j},X_{\be_j+\be_i,j}],\ \ \text{if}\ m_j=0,\\
\frac{1}{(3-m_j)}[X_{\bm-\be_j,j},X_{2\be_j,j}],\ \ \text{if}\ m_j\neq 0, 3,\\
\ [X_{\be_j,i},X_{2\be_j+\be_i,j}]+X_{2\be_j+\be_i,i},\ \ \text{if}\ m_j=3, |\bm|=3,\\
\frac{1}{2}[X_{\bm-2\be_j,j},X_{3\be_j,j}],\ \ \text{if}\ m_j=3, |\bm|>3,
\end{array}
   \right.\end{equation} 
 By induction on $|\bm|$ and the bracket (\ref{bracket}), we can see that $X_{\bm,j}\in H_n$ and it is of the form (\ref{x-operator}).
\end{proof}
By Theorem \ref{tensor-iso} and Lemma \ref{H-generator}, we obtain a PBW type basis of $H_n$.

\begin{corollary}\label{basis}The ordered monomials  in $X_{\bm,j}\in H_n$, for $j\in \{1,\dots,n\}$ and $\bm\in \Z_{\geq 0}^n$ with $|\bm|\geq 1, \bm\neq \be_j$, form a basis of $H_n$.
\end{corollary}

Following the original Premet's proof of the PBW theorem for finite $W$-algebras (see \cite{P}), we will show the following algebra isomorphism.
\begin{theorem}\label{endo} We have that $H_n\cong H_n^{\b1}$.
\end{theorem}
\begin{proof} Let $L_n$ be the Lie subalgebra of $W_n$ spanned by $t^{\bm}\p_i$,  $i=1,\dots,n,  \bm\in\Z^n_{\geq 0}$ with $|\bm|\geq 1$. Clearly as a linear space, $Q_\b1\cong U(L_n)$.
 Any $\theta\in \text{End}_{W_n}(Q_{\b1})$ is determined by $\theta(v_{\b1})$ and  $\theta(v_{\b1})\in \text{wh}_{\b1}(Q_\b1)$.
Since $[H_n, \Delta_n]=0$, we can define an algebra homomorphism $$\Theta: H_n\rightarrow H_n^{\b1}=\text{End}_{W_n}(Q_{\b1})^{\text{op}}, X\mapsto \Theta(X),$$ where $\Theta(X)\in\text{End}_{W_n}(Q_{\b1})^{\text{op}}$ such that $\Theta(X)(v_\b1)=Xv_\b1$. By Lemma \ref{H-generator} and Corollary \ref{basis}, $\Theta$ is injective. To show that $\Theta$ is surjective, we need to prove that $\text{End}_{W_n}(Q_{\b1})^{\text{op}}$ is generated by $\Theta(X_{\bm,j})$, for $j\in \{1,\dots,n\}$ and $\bm\in \Z_{\geq 0}^n$ with $|\bm|\geq 1, \bm\neq \be_j$.
For convenience, denote $$\{y_k\mid k\in \Z_{\geq 1}\}:=\{t_i\p_j-t_i\p_i, t^{\bm}\p_j\mid 1\leq i\neq j\leq n, \bm\in \Z_{\geq 0}^n, |\bm|\geq 2\},$$
 and $$\{Y_k\mid k\in \Z_{\geq 1}\}:=\{\Theta(X_{\bm,j})\mid j=1,\dots,n, \bm\in \Z_{\geq 0}^n, |\bm|\geq 1, \bm\neq \be_j\}.$$
 Then $\{y_k\mid k\in \Z_{\geq 1}\}\cup\{h_1,\dots, h_n\}$ is a basis of $L_n$.
 Let $n_k$ be the integer such that
$[E_n, y_k]=n_ky_k$. For an $\tilde{l}=(l_1,l_2,\dots)\in \Z_{\geq 0}^{\infty}$ with only finite nonzero entries and $\br=(r_1,\dots, r_n)\in \Z_{\geq}^n$,
denote
$$\aligned  &\text{wt} (\br+\tilde{l}):= n_1l_1+\cdots+n_ql_q+\cdots,\\
& |\br+\tilde{l}|:= |\br|+(n_1+1)l_1+\cdots+(n_q+1)l_q+\cdots,\\
&y^{\tilde{l}}h^{\br}:=y_{1}^{l_1}y_{2}^{l_2}\cdots y_{q}^{l_q}\cdots h_1^{r_1}\cdots h_n^{r_n}.
\endaligned$$
We say that the weight of  $y^{\tilde{l}}h^{\br}$ is  $\text{wt} (\br+\tilde{l})$, the degree of  $y^{\tilde{l}}h^{\br}$ is  $|\br+\tilde{l}|$. For any $0\neq \theta\in \text{End}_{W_n}(Q_{\b1})$, we can write $$\theta(v_{\b1})=\sum_{|\tilde{l}+\br|\leq d}a_{\tilde{l},\br}y^{\tilde{l}}h^{\br}v_{\b1},$$
 where $d$ is called the degree of $\theta$, and $a_{\tilde{l},\br}\neq 0$ for at least one $(\tilde{l},\br)$ with $|\tilde{l}+\br|= d$. Let $\Lambda^d=\{(\tilde{l},\br)\mid a_{\tilde{l},\br}\neq 0, |\tilde{l}+\br|= d\}$, and $\Lambda^d_{\text{max}}$ be the subset $\Lambda^d$ of all $(\tilde{p},\bs)$ such that $\text{wt} (\bs+\tilde{p})$ assumes its maximum value. This maximal value will be  denoted by $N(\theta)$. We define $(d, N(\theta))$  to be the Whittaker degree of $\theta$.

 \noindent{\bf Claim:} If $(\tilde{p},\bs)\in\Lambda^d_{\text{max}}$, then
 $\bs=0$.

 Suppose the contrary. Then there is a $(\tilde{p},\bs)\in\Lambda^d_{\text{max}}$ such that   $\bs\neq 0$.
 Hence $s_k>0$ for some $k\in\{1,\dots,n\}$.
  From $[\p_k, h^{\br}]=\sum_{i=1}^{r_k}\binom{r_k}{i}h^{\br-i\be_k}$ and $(\p_k-1)v_{\b1}=0$, we have:
 $$\aligned &(\p_k-1)\theta(v_{\b1})=\sum_{|\tilde{l}+\br|\leq d}a_{\tilde{l},\br}[\p_k, y^{\tilde{l}}h^{\br}]v_{\b1}\\
 =& \sum_{|\tilde{l}+\br|\leq d}a_{\tilde{l},\br}[\p_k, y^{\tilde{l}}]h^{\br}v_{\b1}
 +\sum_{|\tilde{l}+\br|\leq d}a_{\tilde{l},\br}y^{\tilde{l}}[\p_k, h^{\br}]v_{\b1}\\
  =& \sum_{|\tilde{l}+\br|\leq d}a_{\tilde{l},\br}[\p_k, y^{\tilde{l}}]h^{\br}v_{\b1}
 +\sum_{|\tilde{l}+\br|\leq d}\sum_{i=1}^{r_k}\binom{r_k}{i}
 a_{\tilde{l},\br}y^{\tilde{l}}h^{\br-i\be_k}v_{\b1}\\
 =& \sum_{|\tilde{l}+\br|\leq d-1}b_{\tilde{l},\br}y^{\tilde{l}}h^{\br}v_{\b1}.
 \endaligned $$
 We can see that   $s_ka_{\tilde{p},\bs}y^{\tilde{p}}h^{\bs-\be_k}v_{\b1}$ is a nonzero term of degree  $d-1$  in the above sum  whose  weight assumes $\text{wt} (\bs+\tilde{p})$. Hence $(\p_k-1)\theta(v_{\b1})\neq 0$, which is a contradiction.

 Let $\theta'=\theta-\sum_{|\tilde{l}|= d}a_{\tilde{l}}Y^{\tilde{l}}$, where  $Y^{\tilde{l}}:=Y_{1}^{l_1}Y_{2}^{l_2}\cdots Y_{q}^{l_q}\cdots$.  Order the elements in $\Z_{\geq 0}^2$ lexicographically.
 By the above claim, $\theta'$ has the Whittaker degree less than $\theta$. By induction on the Whittaker degree of $\theta$, we can conclude that $\text{End}_{W_n}(Q_{\b1})^{\text{op}}$ is generated by
$\{Y_k\mid k\in \Z_{\geq 1}\}$.
\end{proof}

\section{Equivalences between different blocks}

In this section, using $H_n$-modules, we will show that $\Omega_{\b1}^{[\gamma]}$  is  equivalent to $\mathcal{C}^{[\gamma]}_\alpha$ for  some   $\alpha\in \C^n$.

\subsection{Equivalence between $\Omega_{\b1}$ and $H_n\text{-mod}$}

Let $H_n\text{-mod}$ be the category of finite dimensional $H_n$-modules. By Theorems \ref{s-e} and \ref{endo}, the category  $\Omega_{\b1}$ is equivalent to $H_n\text{-mod}$. In this subsection, we give a more essential proof of this equivalence using the tensor product decomposition of $U_\p$.

We define the functor $F: \Omega_{\b1}\rightarrow H_n\text{-mod}$ such that $F(M)=\text{wh}_{\b1}(M)$. Since  $[H_n, \p_{i}]=0$ and that
$\text{wh}_{\b1}(M)$ is finite dimensional,
 $F(M)\in H_n\text{-mod}$. Conversely, for  a finite dimensional $H_n$-module $V$, define the action of $\Delta_n$ on $V$ such that $\p_i v=v$ for any $v\in V$, $i=1,\dots,n$.
 Let $$G(V):=\text{Ind}_{H_nU(\Delta_n)_\p }^{U_{\p}}V=U_{\p}\otimes_{H_nU(\Delta_n)_\p } V.$$ It is easy to see that $G(V)=U(\mh_n)\otimes V$, since $U_{\p} =B_nH_n$ and $B_n$ is
generated by $\p_i, h_i, \p_i^{-1}, i=1,\dots,n$.

The following lemma in \cite{ZL} is a key observation on modules in $\Omega_{\mathbf{\b1}}$.

\begin{lemma}\label{free}
Let $M$ be a module in
$\Omega_{\mathbf{\b1}}$. Then as a vector space,
$M\cong U(\mh_n)\otimes \mathrm{wh}_{\mathbf{1}}(M)$.\end{lemma}

The following proposition reduces modules in $\Omega_{\b1}$ to finite dimensional $H_n$-modules.
\begin{proposition}\label{equ-hw} The category  $\Omega_{\b1}$ is equivalent to $H_n\text{-mod}$.
\end{proposition}

\begin{proof}
 Since $FG(V)=\text{wh}_{\mathbf{\b1}}(U(\mh_n)\otimes V)=V$ for any $V\in H_n\text{-mod} $,
 we have $FG=\text{id}_{H_n\text{-mod} }$.
For any $M\in \Omega_{\mathbf{1}}$, by Lemma \ref{free}, as a vector space, $M\cong U(\mh_n)\otimes \mathrm{wh}_{\b1}(M)$. By the isomorphism $U_{\p}\cong B_n\otimes  H_n$,
we have that $M\cong G(\mathrm{wh}_{\b1}(M))=GF(M)$. So  $GF=\text{id}_{\Omega_{\b1} }$.
\end{proof}

The following corollary gives  a classification of all finite dimensional irreducible $H_n$-modules in terms of finite dimensional  simple $\gl_n$-modules. In $T(A_n^{\b1}, V)$, we identify $1\otimes  V$ with $V$.

\begin{corollary}\label{action-H}
 If $M$ is a finite dimensional irreducible $H_n$-module, then $M\cong \mathrm{wh}_{\b1}(T(A_n^{\b1}, V))=V$, or  $M\cong \mathrm{wh}_{\b1}(\text{im}\pi_i)$, where $i=0,\dots, n-1$,
$V$ is a finite dimensional  simple $\gl_n$-module which is not isomorphic to $k$-th exterior power $V(\delta_k,k)$ of $\C^n$, for $k = 0,\cdots,n$. Moreover, the action of $H_n$ on $M$ satisfies that:
\begin{equation}\label{3.2}\aligned z_{i,j}v=&\   (E_{ij}-E_{ii})v,\\
z_{i,l,k}v=&\ \Big(E_{ll}E_{ii}-E_{ik}E_{ll}-E_{lk}E_{ii}+\delta_{il}E_{il}\\
 &+(\delta_{lk}-\delta_{il})E_{ik}+(\delta_{ik}-\delta_{il})E_{lk}\Big) v,\\
   z_{i,j,j}v=&\  (E_{ij}-E_{ij}E_{jj})v,\\
 z_iv =&\ 2( E_{ii}^3 - 3E_{ii}^2 + 2E_{ii})v,
\endaligned\end{equation}
where the operators $z_{i,j}, z_{i,l,k}, z_i$ are defined in (\ref{x-w-def}), $i,j= 1,\cdots,n, v\in M$.
\end{corollary}

\begin{proof} By the action of $W_n$ on $T(A_n^{\b1}, V)$, see (\ref{action}), we have that
$$\aligned z_{i,j}(1\otimes v)&=   ((t_i\p_j)\p_i\p_j^{-1}-t_i\p_i)(1\otimes v)\\
&= t_i\otimes v+ 1\otimes E_{ij}v-t_i\otimes v-1\otimes E_{ii}v\\
&= 1\otimes (E_{ij}-E_{ii})v,
\endaligned$$
and
$$\aligned
z_{i,l,k}(1\otimes v)
=&\ ((t_it_l\p_k)\p_i\p_l\p_k^{-1}-(t_i\p_k)\p_i\p_k^{-1}(t_l\p_l )\\
  & -(t_l\p_k)\p_l\p_k^{-1}(t_i\p_i)+(t_l\p_l )(t_i\p_i)+\delta_{il}t_i\p_l)(1\otimes v)\\
=&\ t_it_l\otimes v+t_i\otimes E_{lk}v+t_l\otimes E_{ik}v-(t_i\p_k)\p_i\p_k^{-1}(t_l\otimes v+1\otimes E_{ll}v)\\
 &-(t_l\p_k)\p_l\p_k^{-1}(t_i\otimes v+1\otimes E_{ii}v)+(t_l\p_l)(t_i\otimes v+1\otimes E_{ii}v)\\
 &+\delta_{il}(t_i\otimes v+1\otimes E_{il}v)\\
=&\ t_it_l\otimes v+t_i\otimes E_{lk}v+t_l\otimes E_{ik}v-t_i(t_l+\delta_{il})\otimes v-t_i\otimes E_{ll}v\\
&-(t_l+\delta_{il}-\delta_{lk})\otimes E_{ik}v-1\otimes E_{ik}E_{ll}v-t_l(t_i+\delta_{il})\otimes v-t_l\otimes E_{ii}v\\
&-(t_i-\delta_{ik}+\delta_{il})\otimes E_{lk}v-1\otimes E_{lk}E_{ii}v+t_l(t_i+\delta_{il})\otimes v+t_l\otimes E_{ii}v\\
&+t_i\otimes E_{ll}v+1\otimes E_{ll}E_{ii}v+\delta_{il}(t_i\p_l\otimes v+1\otimes E_{il}v)\\
=&\ 1\otimes \big( E_{ll}E_{ii}-E_{ik}E_{ll}-E_{lk}E_{ii}+\delta_{il}E_{il}\\
&+(\delta_{lk}-\delta_{il})E_{ik}+(\delta_{ik}-\delta_{il})E_{lk}\big)v.
\endaligned$$
We can also obtain that
$$ \aligned z_i(1\otimes v)& =\Big((t_i^3\p_i)\p_i^{2}-3(t_i^2\p_i)(t_i\p_i-1)\p_i
+2(t_i\p_i)(t_i\p_i-1)(t_i\p_i-2)\Big)(1\otimes v)\\
&= 1\otimes2( E_{ii}^3 - 3E_{ii}^2 + 2E_{ii})v.
\endaligned$$
This completes the proof.
\end{proof}

\begin{remark} \label{remark4.4}
For a finite dimensional simple $\gl_n$-module $V(\lambda)$,
we denote
$$W(\lambda)=\left\{\begin{array}{l}
V(\lambda), \ \text{if}\ \lambda\neq \delta_i, i=1,\dots, n,\\
\mathrm{wh}_{\b1}(\text{im}\pi^{A_n^{\b1}}_{i-1}), \text{if}\ \lambda= \delta_i, i=1,\dots, n.
 \end{array}
   \right.$$
   By Corollary \ref{action-H}, under the action ($\ref{3.2}$), $W(\lambda)$ exhaust  all finite dimensional  simple $H_n$-modules.
   Note that $W(\delta_0)\cong W(\delta_1)$ since $\pi_0^{A_n^{\b1}}$ is injective.
\end{remark}

By Corollary \ref{action-H}, the actions of $z_{i,i,i}$ and $ z_i$ on $W(\lambda)$ are diagonalizable. We will use the weight sets of  $z_{i,i,i}$ and $ z_i$ to give a subcategory decomposition of $H_n\text{-mod}$.

\begin{lemma}\label{ext} For $\gamma, \lambda\in \Lambda_{+}$, if $\lambda-\gamma\not\in \Z^n$, then $\mathrm{Ext}^1_{H_n}(W(\gamma), W(\lambda))=0$.
\end{lemma}

\begin{proof}  Since $z_{i,i,i}v=(E_{ii}-E_{ii}^2)v$ and  $ z_iv =2( E_{ii}^3 - 3E_{ii}^2 + 2E_{ii})v$ for any $v\in W(\gamma)\cup W(\lambda)$,  both $z_{i,i,i}$ and $ z_i$ act diagonally  on $W(\gamma), W(\lambda)$ for any $i\in \{1,\dots,n\}$.
We consider eigenvalues of $z_{i,i,i}$ and $ z_i$ on  $W(\gamma)$ and $W(\lambda)$.
Let $S_{\lambda+\Z^n}=\{\big(\beta(z_{1,1,1}),\dots,\beta(z_{n,n,n}),\beta(z_1),\dots, \beta(z_n) \big)\mid \beta\in \lambda+\Z^n\}$.
Note that the eigenvalues of $z_{i,i,i}$ and $ z_i$ on $W(\lambda)$ constitute a subset of  $S_{\lambda+\Z^n}$.
 To show that $\text{Ext}^1_{H_n}(W(\gamma), W(\lambda))=0$,
we will prove that $S_{\gamma+\Z^n}\cap S_{\lambda+\Z^n}=\emptyset$ on the contrary.
Suppose that there exist $\br,\bs\in \Z^n$ such that $(\lambda+\br)(z_{i,i,i})=(\gamma+\bs)(z_{i,i,i})$ and $(\lambda+\br)(z_{i})=(\gamma+\bs)(z_{i})$, for any $i$.
Then $\lambda_i+r_i-(\lambda_i+r_i)^2=\gamma_i+s_i-(\gamma_i+r_i)^2$ and
$(\lambda_i+r_i)^3-(\lambda_i+r_i)^2=(\gamma_i+s_i)^3-(\gamma_i+s_i)^2$,
consequently $\lambda_i+r_i=\gamma_i+s_i$ or $\{\lambda_i+r_i,\gamma_i+s_i\}=\{0,1\}$ for any $i$. This is a contradiction, since $\lambda-\gamma\not\in \Z^n$.
\end{proof}

For a $\gamma\in \Lambda_+$, let $H_n^{[\gamma]}\text{-mod}$ be the Serre subcategory of $H_n\text{-mod}$ generated by simple $H_n$-modules $W(\lambda)$, for $\lambda\in [\gamma]$. By Lemma \ref{ext}, we obtain the following decomposition of $H_n\text{-mod}$.

\begin{corollary}We have the subcategory decomposition
$H_n\text{-mod}=\bigoplus_{[\gamma]}H_n^{[\gamma]}\text{-mod}$.
\end{corollary}

By Proposition \ref{equ-hw}, we have the following subcategory equivalence.

\begin{corollary}\label{O-H} For each $\gamma\in \Lambda_+$,  the subcategory $\Omega_{\b1}^{[\gamma]}$  is  equivalent to  $H_n^{\gamma}\text{-mod}$. Moreover we  have  the subcategory decomposition
$\Omega_{\b1}=\bigoplus_{[\gamma]}\Omega_{\b1}^{[\gamma]}$.
\end{corollary}

\subsection{Equivalence between $H_n\text{-mod}$ and a full subcategory of weight $W_n$-modules}
For  a finite dimensional  $H_n$-module $V$, let each $h_{i}$
 act  on it as the scalar $\alpha_i$.
Define the induced $U_\p$-module $$G_1(V)=\text{Ind}_{U(\mh_n)H_n}^{U_\p}V=U_{\p}\otimes_{U(\mh_n)H_n} V.$$ It is clear that $G_1(V)=\C[\p_1^{\pm 1},\dots, \p_n^{\pm 1}]\otimes V$. Restricted to $W_n$,  $G_1(V)$ becomes a $W_n$-module which is still denoted by $G_1(V)$.

\begin{lemma}
Let $V$ be a finite dimensional  $H_n$-module, the action of $W_n$ on $G_1(V)$ satisfies:
\begin{equation}\label{w-mod}\aligned h_k \cdot (\p^{\br}\otimes v)& =(\alpha_k-r_k)(\p^{\br}\otimes v),\\
  \p_{k} \cdot (\p^{\br}\otimes v)& =\p^{\br+\be_k}\otimes v,\\
   t_i\p_{j}\cdot  (\p^{\br}\otimes v)& =\p^{\br+\be_j-\be_i}\otimes z_{i,j} v+(\alpha_i-r_i+1-\delta_{ij})\p^{\br+\be_j-\be_i}\otimes v,\\
t_it_j\p_j\cdot (\p^{\br}\otimes v)
  &= \p^{\br-\be_i}\otimes\Big(z_{i,j,j} +(\alpha_j-r_j)(z_{i,j}+\alpha_i-r_i+1)\Big)v,
  \endaligned\end{equation}
where $ \p^{\br}=\p_{1}^{r_1}\dots \p_{n}^{r_n}, \br\in \Z^n, 1\leq l, j, k\leq n$.
 \end{lemma}

\begin{proof} It is easy to see that
$$\aligned h_k \cdot (\p^{\br}\otimes v)& =(\alpha_k-r_k)(\p^{\br}\otimes v),\\
  \p_{k} \cdot (\p^{\br}\otimes v)& =\p^{\br+\be_k}\otimes v.\endaligned$$
From  $$\aligned t_i\p_j&=z_{i,j}\p_j\p_i^{-1}+h_i\p_j\p_i^{-1},\\
t_it_j\p_j &=z_{i,j,j}\p_i^{-1}   -\delta_{ij}(t_i\p_j)\p_i^{-1}+(t_i\p_j)\p_i\p_j^{-1}(h_j )\p_i^{-1},
    \endaligned$$
    and $[z_{i,j},\p_k]=[z_{i,j,j},\p_k]=0$,
we can obtain that
  $$\aligned
   t_i\p_{j}\cdot  (\p^{\br}\otimes v)& =\ (z_{i,j}\p_j\p_i^{-1}+h_i\p_j\p_i^{-1})\cdot  (\p^{\br}\otimes v)\\
   &=\ \p^{\br+\be_j-\be_i}\otimes z_{i,j} v+(\alpha_i-r_i+1-\delta_{ij})\p^{\br+\be_j-\be_i}\otimes v,\endaligned$$
   and
$$\aligned t_it_j\p_j\cdot (\p^{\br}\otimes v)
=&\ (z_{i,j,j}\p_i^{-1}-\delta_{ij}(t_i\p_j)\p_i^{-1}+(t_i\p_j)\p_i\p_j^{-1}(h_j )\p_i^{-1})\cdot (\p^{\br}\otimes v)\\
=&\  \p^{\br-\be_i}\otimes\big( z_{i,j,j} +(\alpha_j-r_j)(z_{i,j}+\alpha_i-r_i+1)\big)v.
  \endaligned$$
  This completes the proof.
\end{proof}

 Thus we have a functor $G_1$ from
$H_n\text{-mod}$ to the category of weight modules over $W_{n}$. In the following theorem, we show that  each  $H_n^{[\gamma]}\text{-mod}$  is  equivalent to some  block of the cuspidal category.

\begin{theorem}\label{H-W}For each $\gamma\in \Lambda_+$, there exists an $\alpha\in \C^n$ such that $H_n^{[\gamma]}\text{-mod}$  is  equivalent to $\mathcal{C}^{[\gamma]}_\alpha$.
\end{theorem}

\begin{proof}
We choose an $\alpha\in \C^n$ such that $\alpha_i-\gamma_i\not\in \Z, |\alpha|+\gamma_i\not\in \Z$ for any $i\in \{1,\cdots,n\}$.

{\bf Claim.} $G_1(V)\in \mathcal{C}^{[\gamma]}_\alpha$ for any $V\in H_n^{[\gamma]}\text{-mod}$.

We prove this  claim by induction on the length of the $H_n$-module $V$. By Corollary \ref{action-H} and Remark \ref{remark4.4}, any simple $H_n$-module
in  $H_n^{[\gamma]}\text{-mod}$ is isomorphic to $W(\lambda)$, $\lambda\in [\gamma]$. The action of $H_n$ on  $W(\lambda)$ satisfies  (\ref{3.2}). By
 (\ref{3.2}) and (\ref{w-mod}), the action of $W_n$ on $G_1(W(\lambda))$ is such that:
 $$\aligned
t_i\p_j\cdot(\p^{\br}\otimes v)&=\p^{\br+\be_j-\be_i}\otimes(E_{ij}-E_{ii}+\alpha_i-r_i+1)v, \  i\neq j,\\
t_i E_n\cdot(\p^{\br}\otimes v)&
=\p^{\br-\be_i}\otimes\sum_{j=1}^{n}\big((E_{ij}-E_{ij}E_{jj})+(\alpha_j-r_j)(E_{ij}-E_{ii}+\alpha_i-r_i+1)\big)v,
\endaligned$$
where $1\leq i,j \leq n, v\in W(\lambda)$.

 To show $G_1(W(\lambda))\in \mathcal{C}_\alpha^{[\gamma]}$, we need to prove that $\p_i, t_i\p_j, t_iE_n$ act injectively on $G_1(W(\lambda))$, for any
 $1\leq i \neq j \leq n$.  By the second formula in (\ref{w-mod}), $\p_i$ acts injectively on $G_1(W(\lambda))$.
 Note that $W(\lambda)\subset V(\lambda)$ and $V(\lambda)=\oplus_{\mu\in \C^n}V(\lambda)_\mu$,
 where $V(\lambda)_\mu:=\{w \in V(\lambda) \mid E_{kk}w=\mu_kw, k=1\dots, n\}$.
 For any nonzero $v=\sum_{\mu \in I }v_{\mu}\in W(\lambda)$, where $I\subset \C^n$ is a finite subset, $0\neq v_\mu\in V(\lambda)_\mu$, we have that
 $$\aligned
& (E_{ij}-E_{ii}+\alpha_i-r_i+1)v\\
&=\sum_{\mu \in I }E_{ij}v_{\mu}+ \sum_{\mu \in I }(-\mu_i+\alpha_i-r_i+1)v_{\mu}\\
&=\sum_{\mu \in I }v'_{\mu+\be_i-\be_j}+ \sum_{\mu \in I }(-\mu_i+\alpha_i-r_i+1)v_{\mu},
\endaligned$$
 where $v'_{\mu+\be_i-\be_j}:=E_{ij}v_{\mu}\in V(\lambda)_{\mu+\be_i-\be_j}$. Since $I$ is a finite set, we can  choose $\mu'\in I$ such that $\mu'_i$ is minimal.
  By the choice of $\alpha$, we have that $-\mu'_i+\alpha_i-r_i+1\neq 0$.
  So  $(E_{ij}-E_{ii}+\alpha_i-r_i+1)v\neq 0$,
  and hence $t_i\p_j$
acts injectively on $G_1(W(\lambda))$, for any
 $1\leq i \neq j \leq n$.
 Furthermore,
 $$\aligned & \sum_{j=1}^{n}\big((E_{ij}-E_{ij}E_{jj})+(\alpha_j-r_j)(E_{i,j}-E_{i,i}+\alpha_i-r_i+1)\big)v\\
 &=  \sum_{j=1,j\neq i}^{n} \sum_{\mu \in I }(1-\mu_j+\alpha_j-r_j)v'_{\mu+e_i-e_j}\\
 &\ \ \ +\sum_{\mu \in I } (\mu_i-\mu_i^2)v_{\mu}
 +\sum_{\mu \in I } (\alpha_i-r_i)\mu_iv_{\mu}\\
 & \ \ \ +\sum_{\mu \in I } (|\alpha|-|\br|)(-\mu_i+\alpha_i-r_i+1)v_{\mu}\\
 &= \sum_{j=1,j\neq i}^{n} \sum_{\mu \in I }(1-\mu_j+\alpha_j-r_j)v'_{\mu+e_i-e_j}\\
& \ \ \  +\sum_{\mu \in I } (|\alpha|-|\br|+\mu_i)(-\mu_i+\alpha_i-r_i+1)v_{\mu}.
 \endaligned$$

 The choice of $\alpha$  forces that $(|\alpha|-|\br|+\mu_i)(-\mu_i+\alpha_i-r_i+1)\neq 0$. Hence $t_iE_n$ acts injectively on $G_1(W(\lambda))$, for any
 $1\leq i\leq n$. Therefore  $G_1(W(\lambda))\in \mathcal{C}^{[\gamma]}_\alpha$.
 By induction on the length of the $H_n$-module $V$, we can show that  $G_1(V)\in \mathcal{C}_\alpha^{[\gamma]}$ for any $V\in H_n^{[\gamma]}\text{-mod}$.

 We define the functor $F_1: \mathcal{C}^{[\gamma]}_{\alpha} \rightarrow H_n^{[\gamma]}\text{-mod}$ such that $F_1(M)=M_{\alpha}$.
 From $[H_n, \mh_n]=0$,   we have $H_n M_{\alpha}\subset M_{\alpha}$.
 Hence $M_\alpha \in H_n^{[\gamma]}\text{-mod}$. We can check that $G_1F_1=\text{id}_{\mathcal{C}^{[\gamma]}_{\alpha}}$ and $F_1G_1=\text{id}_{H_n^{[\gamma]}\text{-mod}}$. Therefore
  $H_n^{[\gamma]}\text{-mod}$  is  equivalent to $\mathcal{C}^{[\gamma]}_\alpha$.
\end{proof}

By Corollary \ref{O-H} and Theorem \ref{H-W}, we obtain the main result of the present paper.

\begin{theorem}\label{main-th}For each $\gamma\in \Lambda_+$, there exists an $\alpha\in \C^n$ such that  $\Omega_{\b1}^{[\gamma]}$  is  equivalent to $\mathcal{C}^{[\gamma]}_\alpha$.
\end{theorem}
%
%
%
%

\begin{center}
\bf Acknowledgments
\end{center}

\noindent   G.L. is partially supported by NSF of China(Grants 12371026).


\vspace{4mm}
 \noindent G.L.: School of Mathematics and Statistics,
and  Institute of Contemporary Mathematics,
Henan University, Kaifeng 475004, China. Email: liugenqiang@henu.edu.cn

\vspace{0.2cm}

\noindent   \noindent Y.Z.: School of Mathematics and Statistics,
Henan University, Kaifeng 475004, China. Email:  15518585081@163.com
\end{document}